\documentclass[11pt,twoside]{article}

\usepackage{mathrsfs,amsfonts,amsmath,amssymb}
\usepackage{latexsym}
\newtheorem{satz}{Theorem}
\newtheorem{proposition}[satz]{Proposition}
\newtheorem{theorem}[satz]{Theorem}
\newtheorem{lemma}[satz]{Lemma}
\newtheorem{definition}[satz]{Definition}
\newtheorem{corollary}[satz]{Corollary}

\def\no{\noindent}

\def\eps{\varepsilon}
\def\_phi{\varphi}

\def\a{\alpha}

\def\F{{\mathbb F}}

\def\ov{\overline}

\def\C{{\mathbb C}}

\def\E{\mathsf {E}}

\def\Z_N{{\mathbb Z}_N}
\def\Z{{\mathbb Z}}

\def\Gr{{\mathbf G}}

\def\supp{{\rm supp\,}}

\def\G{\Gamma}
\def\FF{\widehat}
\def\c{\circ}

\def\Cf{{\mathcal C}}

\author{Shkredov I.D.}
\title{ On Heilbronn's exponential sum
\footnote{
This work was supported by grant RFFI NN
06-01-00383, 11-01-00759, Russian Government project 11.G34.31.0053,
Federal Program "Scientific and scientific--pedagogical staff of innovative Russia" 2009--2013
and
grant Leading Scientific Schools N 8684.2010.1.}
}
\date{}
\begin{document}
\maketitle

\begin{center}
 Annotation.
\end{center}

{\it \small
    In the paper
    we prove
    a new upper bound for Heilbronn's exponential sum
    and obtain some applications of our result to distribution of Fermat quotients.
}
\\

\section{Introduction}
\label{sec:introduction}

Let $p$ be a prime number.
Heilbronn's exponential sum is defined by
\begin{equation}\label{def:Heilbronn_sum}
    S(a) = \sum_{n=1}^p e^{2 \pi i \cdot \frac{an^p}{p^2} } \,.
\end{equation}

In papers \cite{H}, \cite{H-K} (see also \cite{Yu})
a nontrivial upper bound for the sum $S(a)$
was obtained.

\begin{theorem}
    Let $p$ be a prime, and $a\neq 0 \pmod p$.
    Then
    $$
        |S(a)| \ll p^{\frac{7}{8}} \,.
    $$
\label{t:H-K}
\end{theorem}

Igor Shparlinski asked to the author about the possibility of an improvement of Theorem \ref{t:H-K}.
The main result of the paper is

\begin{theorem}
    Let $p$ be a prime, and $a\neq 0 \pmod p$.
    Then
    $$
        |S(a)| \ll p^{\frac{59}{68}} \log^{\frac{5}{34}} p \,.
    $$
\label{t:main}
\end{theorem}

Heilbronn's exponential  sum is connected
(see e.g. \cite{BFKS}, \cite{Chang_Fermat}, \cite{Lenstra}, \cite{OstShp}, \cite{Shp-FermVal}, \cite{Shp-Ihara})
 with so--called {\it Fermat quotients} defined as
$$
    q(n) = \frac{n^{p-1}-1}{p} \,, \quad n\neq 0 \pmod p \,.
$$
Our main result has some applications
to the distribution of such quotients.

By $l_p$ denote the smallest $n$ such that $q(n) \neq 0 \pmod p$.
In \cite{BFKS} an upper bound for $l_p$ was obtained.

\begin{theorem}
    One has
    $$
        l_p \le (\log p)^{\frac{463}{252} + o(1)}
    $$
    as $p\to \infty$.
\end{theorem}

We improve the constant $\frac{463}{252}$ above, see Theorem \ref{t:Fermat_quatients} of section \ref{sec:proof}.



Another applications are : \\
$\bullet~$ discrepancy of Fermat quotients from \cite{OstShp},  Theorems 18--19, \\
$\bullet~$ new bound for the size of the image of $q(n)$, see \cite{Shp-FermVal}, Theorem 1, \\
$\bullet~$ estimates for Ihara sum, \cite{Shp-Ihara}, \\
$\bullet~$ better bounds for the sums $\sum_{n=1}^k \chi(q(n))$, $\sum_{n=1}^k \chi(nq(n))$,
see \cite{Chang_Fermat}, Theorems 3.1, 3.2, 5.4.

\bigskip

Let us say few  words about the proof.
Clearly, sum (\ref{def:Heilbronn_sum}) can be considered as the sum over
the following multiplicative subgroup
\begin{equation}\label{def:H_Gamma}
    \G = \{ m^p ~:~ 1\le m \le p-1 \} \subseteq \Z/(p^2 \Z) 
\end{equation}
(see the discussion at the beginning of section \ref{sec:preliminaries}).
Recently, some progress in estimating of  exponential sums over "large"\, subgroups
(but in $\Z/p\Z$ not in $\Z/p^2\Z$)
such as
(\ref{def:H_Gamma})
was attained (see \cite{s_ineq}).
So it is natural to try to use the approach from the paper to obtain some new upper bound for (\ref{def:Heilbronn_sum}).
Unfortunately, the methods from \cite{s_ineq} cannot be applied directly in the case.
The reason is that we know much less about
distribution of Heilbronn's subgroup (\ref{def:H_Gamma})
then about
subgroups in $\Z/p\Z$
as well as about looking similar convex--type
sets
(see sections 6, 7 from \cite{s_ineq}).
The only we know is Lemma \ref{l:H-K_2/3} below, which gives, roughly speaking,
a nontrivial
upper bound for the number of the solutions of the equation $x-y \equiv c \pmod {p^2}$
for fixed $c \in \Z/(p^2 \Z)$, $c\neq 0$ and $x,y\in \G$
as well as upper bounds for the moments of such quantities.
Nevertheless the size of $\G$ is large and the ordinary Fourier transformation methods
(see Lemma \ref{l:G-inv_bound_F}), combining with the approach from \cite{s_ineq},
namely, so--called the eigenvalues method
allows us to prove Theorem \ref{t:main}.


The author is grateful to
Igor Shparlinski
for useful discussions as well as pointing some applications of the main result.

\section{Definitions}
\label{sec:definitions}

Let $\Gr$ be an abelian group.
If $\Gr$ is finite then denote by $N$ the cardinality of $\Gr$.
It is well--known~\cite{Rudin_book} that the dual group $\FF{\Gr}$ is isomorphic to $\Gr$ in the case.
Let $f$ be a function from $\Gr$ to $\mathbb{C}.$  We denote the Fourier transform of $f$ by~$\FF{f},$
\begin{equation}\label{F:Fourier}
  \FF{f}(\xi) =  \sum_{x \in \Gr} f(x) e( -\xi \cdot x) \,,
\end{equation}
where $e(x) = e^{2\pi i x}$.
We rely on the following basic identities
\begin{equation}\label{F_Par}
    \sum_{x\in \Gr} |f(x)|^2
        =
            \frac{1}{N} \sum_{\xi \in \FF{\Gr}} \big|\widehat{f} (\xi)\big|^2 \,.
\end{equation}
\begin{equation}\label{svertka}
    \sum_{y\in \Gr} \Big|\sum_{x\in \Gr} f(x) g(y-x) \Big|^2
        = \frac{1}{N} \sum_{\xi \in \FF{\Gr}} \big|\widehat{f} (\xi)\big|^2 \big|\widehat{g} (\xi)\big|^2 \,.
\end{equation}
and
\begin{equation}\label{f:inverse}
    f(x) = \frac{1}{N} \sum_{\xi \in \FF{\Gr}} \FF{f}(\xi) e(\xi \cdot x) \,.
\end{equation}
If
$$
    (f*g) (x) := \sum_{y\in \Gr} f(y) g(x-y) \quad \mbox{ and } \quad (f\circ g) (x) := \sum_{y\in \Gr} f(y) g(y+x)
$$
 then
\begin{equation}\label{f:F_svertka}
    \FF{f*g} = \FF{f} \FF{g} \quad \mbox{ and } \quad \FF{f \circ g} = \FF{f}^c \FF{g} = \ov{\FF{\ov{f}}} \FF{g} \,,
\end{equation}
where for a function $f:\Gr \to \mathbb{C}$ we put $f^c (x):= f(-x)$.
 Clearly,  $(f*g) (x) = (g*f) (x)$ and $(f\c g)(x) = (g \c f) (-x)$, $x\in \Gr$.

We use in the paper  the same letter to denote a set
$S\subseteq \Gr$ and its characteristic function $S:\Gr\rightarrow \{0,1\}.$
Write $\E(A,B)$ for the {\it additive energy} of two sets $A,B \subseteq \Gr$
(see e.g. \cite{tv}), that is
$$
    \E(A,B) = |\{ a_1+b_1 = a_2+b_2 ~:~ a_1,a_2 \in A,\, b_1,b_2 \in B \}| \,.
$$
If $A=B$ we simply write $\E(A)$ instead of $\E(A,A).$
Clearly,
\begin{equation}\label{f:energy_convolution}
    \E(A,B) = \sum_x (A*B) (x)^2 = \sum_x (A \circ B) (x)^2 = \sum_x (A \circ A) (x) (B \circ B) (x)
    \,.
\end{equation}
and by (\ref{svertka}),
\begin{equation}\label{f:energy_Fourier}
    \E(A,B) = \frac{1}{N} \sum_{\xi} |\FF{A} (\xi)|^2 |\FF{B} (\xi)|^2 \,.
\end{equation}

Let
\begin{equation}\label{f:E_k_preliminalies}
    \E_k(A)=\sum_{x\in \Gr} (A\c A)(x)^k \,,
\end{equation}
and
\begin{equation}\label{f:E_k_preliminalies_B}
\E_k(A,B)=\sum_{x\in \Gr} (A\c A)(x) (B\c B)(x)^{k-1}
\end{equation}
be the higher energies of $A$ and $B$.
Similarly, we write $\E_k(f,g)$ for any complex functions $f$ and $g$.
Quantities $\E_k (A,B)$ can be written in terms of generalized convolutions (see \cite{ss_E_k}).

\begin{definition}
   Let $k\ge 2$ be a positive number, and $f_0,\dots,f_{k-1} : \Gr \to \C$ be functions.
Write $F$ for the vector $(f_0,\dots,f_{k-1})$ and $x$ for vector $(x_1,\dots,x_{k-1})$.
Denote by
${\mathcal C}_k (f_0,\dots,f_{k-1}) (x_1,\dots, x_{k-1})$
the function
$$
    \Cf_k(F) (x) =  \Cf_k (f_0,\dots,f_{k-1}) (x_1,\dots, x_{k-1}) = \sum_z f_0 (z) f_1 (z+x_1) \dots f_{k-1} (z+x_{k-1}) \,.
$$
Thus, $\Cf_2 (f_1,f_2) (x) = (f_1 \circ f_2) (x)$.
If $f_1=\dots=f_k=f$ then write
$\Cf_k (f) (x_1,\dots, x_{k-1})$ for $\Cf_k (f_1,\dots,f_{k}) (x_1,\dots, x_{k-1})$.
\end{definition}

For a positive integer $n,$ we set $[n]=\{1,\ldots,n\}$.
All logarithms used in the paper are to base $2.$
By  $\ll$ and $\gg$ we denote the usual Vinogradov's symbols.
If $N$ is a
positive integer then write $\Z_N$ for $\Z/N\Z$ and
if $p$ is a prime then put
$\Z^*_p$ for $(\Z/p\Z) \setminus \{ 0 \}$.

\section{Preliminaries}
\label{sec:preliminaries}


Put
$$
    \G = \{ m^p ~:~ 1\le m \le p-1 \} \subseteq \Z_{p^2} \,.
$$
It is easy to see that $\G$ is a subgroup and that
$$
    \G = \{ x^p ~:~ x\in \Z_{p^2} \,,~ x\neq 0 \} = \{ x \in \Z_{p^2} ~:~ x^p \equiv 1 \pmod {p^2} \}
$$
because of $x\equiv y \pmod p$ implies $x^p \equiv y^p \pmod {p^2}$.

\bigskip

Put
$$
    f(x) = x + \frac{x^2}{2} + \frac{x^3}{3} + \dots + \frac{x^{p-1}}{p-1} \in \Z_p [x] \,.
$$
Put also
$$
    \mathcal{F} (u) = |\{ x\in \Z_p ~:~ f(x) = u \}| \,.
$$
We prove a simple lemma which is connecting the numbers $\mathcal{F} (u)$ and the convolutions of the subgroup $\G$.

\begin{lemma}
    Let $0 \le a,b \le p-1$.
    Then
\begin{displaymath}
(\G \c \G) (a + b p) =\left\{\begin{array}{ll}
\mathcal{F} (aq(a)-b)\,, &\textrm{if~~} a \neq 0\,,\\
|\G|\,, &\textrm{if~~} a=b=0\,, \\
0\,,    &\textrm{otherwise.}
\end{array}\right.
\label{l:G_via_f}
\end{displaymath}
\end{lemma}
\begin{proof}
To calculate $(\G \c \G) (a + b p)$ consider the equation
\begin{equation}\label{f:28.08.2012_1}
    m_1^p - m_2^p \equiv a+bp \pmod{p^2} \,,
\end{equation}
where $m_1,m_2\in \G$.
From (\ref{f:28.08.2012_1}) it follows that  $m_1-m_2 \equiv a \pmod p$.
If $a\equiv 0 \pmod p$ then $b\equiv 0 \pmod p$.
Suppose that $a\neq 0 \pmod p$.
Then
$$
    m_1 \equiv av \pmod p \quad \mbox{ and } \quad m_2 \equiv a(v-1) \pmod p
$$
for some $v\neq 0 \pmod p$.
It follows that
$$
    m_1^p - m_2^p \equiv a^p (v^p - (v-1)^p) \equiv a^p \cdot \sum_{l=1}^p (-1)^{l-1} v^{p-l} \binom{p}{l}
        \equiv
            a^p \cdot (1-p f(v)) \pmod{p^2}
$$
and hence
$$
    a+bp \equiv a^p \cdot (1-p f(v)) \pmod{p^2}
$$
This completes the proof.
$\hfill\Box$
\end{proof}

\bigskip

The following lemma was proved in \cite{H-K}.

\begin{lemma}
    Let $U\subseteq \Z_p$ be a set.
    Then
    $$
        \sum_{u\in U} \mathcal{F} (u) \ll p^{2/3} |U|^{2/3} \,.
    $$
\label{l:H-K_2/3}
\end{lemma}

Using Lemma \ref{l:G_via_f} and Lemma \ref{l:H-K_2/3},
 one can easily deduce upper bounds for moments of convolution of $\G$.
 These estimates are the same as in the case of multiplicative subgroups in $\Z_p$ (see, e.g. \cite{ss}).

\begin{corollary}
    We have
    \begin{equation}\label{f:E_2_E_3}
        \E(\Gamma) \ll |\Gamma|^{5/2} \,, \quad \E_3 (\G) \ll |\G|^3 \log |\G| \,,
    \end{equation}
    and for all $l\ge 4$ the following holds
    \begin{equation}\label{f:E_l}
        \E_l (\G) = |\G|^l + O(|\G|^{\frac{2l+3}{3}}) \,.
    \end{equation}
\label{cor:E_l}
\end{corollary}

We need in a
lemma about Fourier coefficients of an arbitrary set
which is invariant under the  action of a subgroup
(see e.g. \cite{ss}).

\begin{lemma}
        Let $\G \subseteq \Z_{N}$ be a multiplicative subgroup,
        and $Q$ be an $\G$--invariant subset of $\Z_{N}$,
        that is $Q\G=Q$.
        Then for any $\xi \neq 0$ the following holds
\begin{equation}\label{f:G-inv_bound_F}
    | \FF{Q} (\xi) | \le \min \left\{ \left(\frac{|Q|N}{|\G|}\right)^{1/2} \,, \frac{|Q|^{3/4} N^{1/4} \E^{1/4} (\G)}{|\G|} \,,
                            N^{1/8} \E^{1/8} (\G) \E^{1/8} (Q) \left(\frac{|Q|}{|\G|}\right)^{1/2} \right\} \,.
\end{equation}
\label{l:G-inv_bound_F}
\end{lemma}

Finally, we formulate à lemma from \cite{s_ineq}, see Theorem 57, section 7.
This is the key new ingredient of our proof.

\begin{lemma}
    Let $A, D\subseteq \Gr$ be two sets, and $D=-D$.
    Then
$$
    \mu^2 \cdot \E(A,f) \le \sum_{x,y,z\in A} D (x-y) D (x-z) (A\c A) (y-z) \,,
$$
where
$$
    \mu \ge \frac{1}{|A|} \sum_{x\in D} (A\c A) (x) \,,
$$
and $f$ is a nonnegative function such that $\| f\|_2 = 1$, $\supp f \subseteq A$, and
$$
    \mu f(x) = A(x) (D * f) (x) \,.
$$
\label{l:ineq}
\end{lemma}

\section{The proof of the main result}
\label{sec:proof}


Let $\G$ be the subgroup from (\ref{def:H_Gamma}).

\begin{theorem}
    We have
\begin{equation}\label{f:Heilbronn_new}
    \E (\G) \ll |\G|^{\frac{42}{17}} \log^{\frac{10}{17}} |\G| \,.
\end{equation}
\label{t:Heilbronn_new}
\end{theorem}
\begin{proof}
Let
$P=p^2$,
$t=|\G|$, $\E = \E(\G) = t^3/K$, $K\ge 1$, and $\E_3 = \E_3 (\G)$.
By Lemma \ref{l:H-K_2/3} and simple average arguments, we have
\begin{equation}\label{tmp:27.08.2012_1}
    \E \ll \sum_{x \neq 0 ~:~ 2^{-1} tK^{-1} < (\G \c \G) (x) \le cK } (\G\c \G)^2 (x) \,,
\end{equation}
where $c>0$ is some absolute constant.
Let
$$
    D_j = \{ x \neq 0 ~:~ c2^{-j} K < (\G \c \G) (x) \le c2^{-j+1} K \} \,, \quad j\le \log (2cK^2 t^{-1}) \,.
$$
Put $l= [\log (2c K^2 t^{-1})]$.
By (\ref{tmp:27.08.2012_1}) there is $j \in [l]$ such that
\begin{equation}\label{tmp:27.08.2012_2}
    \frac{2^{j} \E}{lK} \ll \sum_{x\in D_j} (\G \c \G) (x) \,.
\end{equation}
Put $D=D_j$.
Clearly, $D=-D$ and $D$ is $\G$--invariant set.
Note also
\begin{equation}\label{f:D_upper}
    |D| \ll \frac{2^{j} \E}{t} \,.
\end{equation}
Put
$$
    \sigma := \sum_{x\in D} (\G \c \G) (x)
$$
Using Lemma \ref{l:ineq} with $A=\G$, $D=D$ and also inequality (\ref{tmp:27.08.2012_2}), we obtain
$$
    \frac{2^{2j} \E^2}{l^2 K^2 t^2} \cdot \E(\G,f)
        \ll
                \frac{\sigma^2}{t^2} \cdot \E(\G,f)
                    \ll
            \mu^2 \cdot \E(\G,f)
                \le \sum_{x,y,z\in \G} D (x-y) D (x-z) (\G\c \G) (y-z) \,,
$$
where
\begin{equation}\label{f:eigen_G_D}
    \mu f(x) = \G(x) (D * f) (x) \,.
\end{equation}
Because of $D$ is $\G$--invariant it is easy to
see that any solution $f$ of equation (\ref{f:eigen_G_D}) coincide with a character on $\G$.
We know that $f(x) \ge 0$, so $f$
is
the main character, i.e.
$f(x) =\G(x)/ t^{1/2}$
(for more details see \cite{ss_E_k} or \cite{s_ineq}).
Thus
\begin{equation}\label{tmp:27.08.2012_3}
    \frac{2^{2j} \E^5}{l^2 t^9} = \frac{2^{2j} \E^3}{l^2 K^2 t^3}
            \ll
                \frac{\sigma^2}{t^2} \cdot \E(\G,f)
        \ll \sum_{x,y,z\in \G} D (x-y) D (x-z) (\G\c \G) (y-z) \,.
\end{equation}
If $y=z$ then by (\ref{tmp:27.08.2012_2}), (\ref{tmp:27.08.2012_3})
and the definition of $\sigma$ the following holds
$$
    \frac{2^{2j} \E^5}{l^2 t^9}
        \ll
            \frac{2^{j} \E^3}{l t^6} \sigma
                \ll
            t \sigma
$$
and we are done.
Thus
$$
    \frac{2^{2j} \E^5}{l^2 t^9}
        \ll
            \sum_{\a \neq \beta} D (\a) D (\beta) (\G\c \G) (\a-\beta) \Cf_3 (\G) (\a,\beta) \,.
$$
Using Cauchy--Schwarz,
 we get
\begin{equation}\label{tmp:26.08.2012_0}
   \frac{2^{4j} \E^{10}}{l^4 t^{18} \E_3}
    \ll
        \sum_{\a \neq \beta} D (\a) D (\beta) (\G\c \G)^2 (\a-\beta) \,.
\end{equation}
Put
$$
    S_i = \{ x\neq 0 ~:~ c' t^{2/3} 2^{-i} < (\G\c \G) (x) \le c' t^{2/3} 2^{-i+1} \} \,,
$$
where $c'>0$ is an absolute constant such that $(\G\c \G) (x) \le c' t^{2/3}$ for all $x\neq 0$.
Such constant exists by Lemma \ref{l:H-K_2/3}.
Trivially
\begin{equation}\label{tmp:27.08.2012_4}
    |S_i| t^{2/3} 2^{-i} c' \le t^2 \,,
\end{equation}
and
\begin{equation}\label{tmp:27.08.2012_4'}
    |S_i| t^{4/3} 2^{-2i} (c')^2 \le \E \,.
\end{equation}
By the definition of the sets $S_i$, we have
$$
    \frac{2^{4j} \E^{10}}{l^4 t^{18} \E_3}
        \ll
            t^{4/3} \sum_i 2^{-2i} \sum_x S_i(x) (D\c D) (x) \,.
$$
Certainly, each set $S_i$ is $\G$--invariant.
Thus, using Lemma \ref{l:G-inv_bound_F}, Fourier transform and Parseval, we obtain
\begin{equation}\label{tmp:27.08.2012_5-}
    \frac{2^{4j} \E^{10}}{l^4 t^{18} \E_3}
        \ll
            t^{1/3} \sum_i 2^{-2i} |S_i|^{3/4} |D| P^{1/4} \E^{1/4} + t^{4/3} \sum_i 2^{-2i} \frac{|S_i||D|^2}{P} \,.
\end{equation}
Let us estimate the second term from (\ref{tmp:27.08.2012_5-}).
Using (\ref{tmp:27.08.2012_4'}) and (\ref{f:D_upper}), we get
$$
    t^{4/3} \sum_i 2^{-2i} \frac{|S_i||D|^2}{P}
        \ll
            \frac{2^{2j} \E^3}{t^2 P} \,.
$$
If
$$
    \frac{2^{4j} \E^{10}}{l^4 t^{18} \E_3}
        \ll
            \frac{2^{2j} \E^3}{t^2 P}
$$
then $\E \ll t^{17/7} \log^{5/7} t$ and the result follows.
Thus
\begin{equation}\label{tmp:27.08.2012_5}
    \frac{2^{4j} \E^{10}}{l^4 t^{18} \E_3}
        \ll
            t^{1/3} \sum_i 2^{-2i} |S_i|^{3/4} |D| P^{1/4} \E^{1/4} \,.
\end{equation}
By Lemma \ref{l:H-K_2/3} and
inequalities
(\ref{tmp:27.08.2012_4}), (\ref{tmp:27.08.2012_4'}), we have
$$
    |S_i| \ll \min \{t2^{3i},  \E t^{-4/3} 2^{2i}, t^{4/3} 2^i \} \,.
$$
Applying the first  bound for $2^i \ll \E t^{-7/3}$,
the second one for $\E t^{-7/3} \ll 2^i \ll t^{8/3} \E^{-1} $ and the third
bound
for other $i$,
we get by (\ref{tmp:27.08.2012_5})
$$
    \frac{2^{4j} \E^{10}}{l^4 t^{18} \E_3}
        \ll
            t^{1/3} |D| P^{1/4} \E^{1/4}
                \left( t^{3/4} \left( \frac{\E}{t^{7/3}} \right)^{1/4}
                        + \frac{\E^{3/4}}{t} \cdot \left( \frac{t^{7/3}}{\E} \right)^{1/2}
                        + t \left( \frac{\E}{t^{8/3}} \right)^{5/4} \right)
                        =
$$
$$
    =
        t^{1/3} |D| P^{1/4} \E^{1/4} \left( \E^{1/4} t^{1/6} + \E^{5/4} t^{-7/3} \right)
            =
                t^{1/2} |D| P^{1/4} \E^{1/2} \left( 1 + \frac{\E}{t^{5/2}} \right)
                    \ll
                        t^{1/2} |D| P^{1/4} \E^{1/2} \,,
$$
where Corollary \ref{cor:E_l} has been used.
Applying the last bound and inequality (\ref{f:D_upper}) after some calculations,
 we obtain (\ref{f:Heilbronn_new}).
This completes the proof.
$\hfill\Box$
\end{proof}

\begin{corollary}
    Let $p$ be a prime, $a\neq 0 \pmod p$,
    and $M,N$ be positive integers, $N\le p$.
    Then
    \begin{equation}\label{f:main_progr}
        \left| \sum_{n=M}^{N+M} e\left( \frac{an^p}{p^2} \right) \right|
            \ll
                p^{\frac{21}{34}} N^{\frac{1}{4}} \log^{\frac{5}{34}} p \,.
    \end{equation}
    In particular
    \begin{equation}\label{f:main_S}
        |S(a)| \ll p^{\frac{59}{68}} \log^{\frac{5}{34}} p \,.
    \end{equation}
\label{cor:main}
\end{corollary}
\begin{proof}
One can get (\ref{f:main_S}) just using Theorem \ref{t:Heilbronn_new} and Lemma \ref{l:G-inv_bound_F}.
To obtain (\ref{f:main_progr}) write $P = [M,\dots,N+M]$ and note that
by Fourier transform or the completing method
(see \cite{H} or \cite{H-K} for details) combining with  H\"{o}lder, we get
\begin{equation}\label{tmp:27.08.2012_10}
    \left| \sum_{n=M}^{N+M} e\left( \frac{an^p}{p^2} \right) \right|
        \le
            \frac{1}{p} \sum_x |\FF{\G} (x)| |\FF{P} (x)|
                \le
                    \E^{1/4} (\G) \cdot \left( \frac{1}{p} \sum_x |\FF{P} (x)|^{4/3} \right)^{3/4} \,.
\end{equation}
By assumption $N\le p$.
Using Theorem \ref{t:Heilbronn_new}  to estimate $\E(\G)$ and a well--known estimate
(see e.g. \cite{H-K})
$$
    \sum_x |\FF{P} (x)|^{4/3} \ll p N^{1/3} \,,
$$
we have (\ref{f:main_progr}).
This completes the proof.
$\hfill\Box$
\end{proof}

\bigskip

Using the arguments from \cite{BFKS} and Theorem \ref{t:Heilbronn_new},
we obtain the following result
about Fermat quotients.

\begin{theorem}
        One has
    $$
        l_p \le (\log p)^{\frac{7829}{4284} + o(1)}
    $$
    as $p\to \infty$.
\label{t:Fermat_quatients}
\end{theorem}

Note that
$$
    \frac{463}{252} = 1.83730 \dots \quad \quad \mbox{ and} \quad \quad \frac{7829}{4284} = 1.82749 \dots
$$

\bigskip

Theorem \ref{t:Fermat_quatients}
has a consequence (see \cite{Lenstra}).

\begin{corollary}
    For every $\eps > 0$ and a sufficiently large integer $n$,
    if $a^{n-1} \equiv 1 \pmod n$
    for every positive integer $a \le (\log p)^{\frac{7829}{4284} + \eps}$
     then $n$ is squarefree.
\end{corollary}

Similar improvement of constant $\frac{463}{252}$ of Theorem 1 from \cite{Shp-FermVal}
as well further applications from \cite{Shp-Ihara}
can be obtained exactly the same way.
New values of constants in paper
\cite{Chang_Fermat} follows from the proposition below
(compare with Proposition 2.1 of article \cite{Chang_Fermat}) and our Theorem \ref{t:Heilbronn_new}.

\begin{proposition}
    For $\xi \in \Z_p$ define
    $$
        u(\xi) = |\{ x\in [p] ~:~ x^p-x \equiv \xi p \pmod {p^2} \}| \,.
    $$
    Then for any $\eps>0$ and all sufficiently large $p$ the following holds
    $$
        \sum_\xi u^2 (\xi)
                \le
            p^{\frac{1}{8}+\eps} \E^{1/2} (\G) \,.
    $$
\end{proposition}



\bigskip

\no{Division of Algebra and Number Theory,\\ Steklov Mathematical
Institute,\\
ul. Gubkina, 8, Moscow, Russia, 119991}
\\
and
\\
Delone Laboratory of Discrete and Computational Geometry,\\
Yaroslavl State University,\\
Sovetskaya str. 14, Yaroslavl, Russia, 150000
\\
and
\\
IITP RAS,  \\
Bolshoy Karetny per. 19, Moscow, Russia, 127994\\
{\tt ilya.shkredov@gmail.com}

\end{document}